\documentclass[conference] {IEEEtran}
\usepackage[utf8]{inputenc}
\usepackage{amsmath}
\usepackage{amssymb}

\title{Selecting Optimal Sampling Rate for Stable Super-Resolution}
\author{
    Nuha Diab \\
    {\small Department of Applied Mathematics, Tel-Aviv University, Tel-Aviv, Israel} \\
    {\small nuhadiab@tauex.tau.ac.il}
}
\usepackage{etex}
\usepackage[T1]{fontenc}
\usepackage{amsmath}
\usepackage{amsthm}
\usepackage{amssymb}
\usepackage{float}
\usepackage{courier}
\usepackage{subcaption}
\usepackage{comment}
\usepackage{etex}
\usepackage[T1]{fontenc}
\usepackage{amsmath}
\usepackage{amsthm}
\usepackage{amssymb}
\usepackage{float}
\usepackage{courier}
\usepackage{comment}
\usepackage{graphicx}
\usepackage{xcolor}

\usepackage{algorithm}
\usepackage{algpseudocode}
\usepackage[normalem]{ulem}
\newtheorem{theorem}{Theorem}
\newtheorem{lemma}{Lemma}
\newtheorem{definition}{Definition}
\newtheorem{remark}{Remark}

\newtheorem{proposition}{Proposition}

\begin{document}

\maketitle

\begin{abstract}
    We investigate the recovery of nodes and amplitudes from noisy frequency samples in spike train signals, also known as the super-resolution (SR) problem. When the node separation falls below the Rayleigh limit, the problem becomes ill-conditioned. Admissible sampling rates, or decimation parameters, improve the conditioning of the SR problem, enabling more accurate recovery. We propose an efficient preprocessing method to identify the optimal sampling rate, significantly enhancing the performance of SR techniques.
\end{abstract}
\section{Introduction}
\par Consider the following "spike train" signal:
\begin{equation}\label{spike}
    \mu (x) := \sum_{i=1}^n a_j\delta_{x_j}, \quad a_j \in \mathbb{C}, \quad x_j \in \mathbb{T}:=
    \mathbb{R}\mod{2\pi}.
\end{equation}
where $\delta_x$ is the Dirac $\delta$-distribution. We refer to $\{x_j\}_{j=1}^n$ as the nodes and to $\{a_j\}_{j=1}^n$ as the coefficients.
Given noisy band-limited Fourier measurements $\hat{\mu}_{\epsilon}(\omega) = \hat{\mu}(\omega) + e(\omega) $, where
\begin{equation}\label{eq:samples}
    \hat{\mu}(\omega) = \sum_{j=1}^na_j\exp{(i\omega x_j)} , \quad \omega \in [-\Omega,\Omega]
\end{equation}
where $\Omega > 0$ and $\|e(\omega)\|_{\infty} \leq \epsilon$ for some $\epsilon > 0$, the goal is to recover $\{a_j, x_j\}_{j=1}^n$ in equation \eqref{spike}.
This \textbf{Super-resolution} (SR) problem is an inverse problem of great theoretical and practical interest, with diverse applications in optics, imaging, inverse scattering, signal processing, spectroscopy, and data analysis in general \cite{donoho1992a, bertero1996iii,lindberg2012mathematical, ammari2005music, stoica2005spectral,mulleti2017super,bhandari2016signal}. 

\par Let the SR factor be defined as $SRF:=(\Delta\Omega)^{-1}$, where $\Delta$ is the smallest separation of the nodes. When $SRF \gg 1$, the SR problem becomes very ill-conditioned. In this regime, the nodes form 'clusters'. An important goal is to prove that certain algorithms are optimal, attaining the optimal reconstruction bounds known as the Min-max error bounds \cite{Batenkov2021}. Extensive research has been conducted on this model regarding its stability analysis and algorithms \cite{li2020super,liu2021theory,derevianko2023esprit,batenkov2023super,demanet2015recoverability,morgenshtern2016super,diederichs2018sparse,kunis2018condition,fan2023efficient}.

\par The utilization of the Decimation technique was pioneered in \cite{Batenkov2021} to establish bounds on the Min-max error in the SR regime. The decimation method proceeds as follows: the spectral range $[-\Omega, \Omega]$ is uniformly sampled at a rate $\rho$, and subsequently, a system of equations of the form $\hat{\mu}(\rho k) = \sum_{j=1}^n a_j \exp(i\rho kx_j)$ is obtained. However, not all $\rho$'s are desirable or suitable; some may result in collisions and the formation of new clusters, thus worsening the ill-conditioned nature of the problem. The existence of admissible $\rho$'s within a certain interval was first proved in \cite{Batenkov2021}. An optimal approach proposed in \cite{Batenkov2021} involves employing an oracle to obtain the correct value of the decimation parameter $\rho$.
\par The Decimation method has already been applied in SR algorithms, such as the Decimated Prony Method (DP) \cite{katz2023decimated} and VEXPA \cite{Briani2020}, to improve the conditioning of the problem (see also \cite{plonka2021deterministic}). Using decimation as a preprocessing step in any SR method with a suitable $\rho$ yields a set of solutions $\{e^{\imath \rho x_j}\}_{j=1}^n$ instead of $\{e^{\imath x_j}\}_{j=1}^n$. This new set of solutions forms a well-separated configuration compared to the undecimated one, making recovery easier and more accurate.
\par Furthermore, the optimality of DP is established due to the optimality of Prony's method for $\Omega=2n$ combined with decimation \cite{katz2023accuracy}. \textit{Despite its practical usage, there is currently no mathematical evidence supporting the identification of such admissible decimation parameter}.
\par In this paper, our objective is to develop a method to automatically select an optimal decimation parameter \( \rho \) within a given interval, ensuring effective separation between the nodes.
Towards that goal, we find the spectrum of the square Vandermonde matrix of the samples \eqref{eq:samples}, showing its dependence on \( \Delta \) and on the number of clusters the nodes form (similarly to \cite{batenkov2021spectral} for rectangular Vandermonde matrices with number of samples $N \gg 2n$). Consequently, we find the spectrum of the Toeplitz matrix of the samples \eqref{teoplitz}, which is our main tool to select the optimal decimation parameter. Lastly, we present the Enhanced Decimated Prony's method (EDP) and the Decimated Matrix Pencil method (DMP), demonstrating their time-complexity advantage over both DP and the Matrix Pencil Method (MP) \cite{hua1990matrix} and then show the numerical optimality of EDP.

\section{Preliminaries}
We recall some definitions from \cite{batenkov2021spectral}.
\begin{definition}
    For $x \in \mathbb{R}$, we denote
    \[ \|x\|_{\mathbb{T}}:=\big|\arg{(e^{\imath x})}\big| = \big|x \mod{(-\pi,\pi]}\big|,\]
    where $\arg(z)$ is the principal value of the argument of $z \in \mathbb{C}\setminus\{0\}$, taking values in $(-\pi,\pi]$.
\end{definition}
For a set of $n$ distinct nodes $X:=\{x_j\}_{j}^n$ with $x_j \in (-{\pi \over 2},{\pi \over 2}]$, we introduce the following definitions.
\begin{definition}\label{def:min-sep}
    Define the \textbf{minimal separation} of the set $X$ as
    \[\Delta = \Delta(X) := \min_{i\neq j}\|x_i-x_j\|_{\mathbb{T}}.\]
    In addition, for any $\rho \in \mathbb{R}$ we define 
     \[\Delta_{\rho} = \Delta(\rho X) := \min_{i\neq j}\|\rho x_i-\rho x_j\|_{\mathbb{T}}.\]
\end{definition}
\begin{definition}[Single cluster configuration]
    The set of nodes $X$ is said to form an $(\Delta, \nu, n)$-cluster if 
    \[
        \forall x,y \in X, x\neq y: \Delta \leq \|y-x\|_{\mathbb{T}} \leq \nu\Delta.
    \]
\end{definition}
\begin{definition}[Multi-cluster configuration]\label{def:multi-cluster}
    The set of nodes $X$ is said to form an $((h^{(j)}, \nu^{(j)}, n^{(j)})_{j=1}^M, \eta)$-clustered configuration if there exist an $M$-partition $X=\bigcup_{j=1}^M\mathcal{C}^{(j)}$, such that for each $j\in \{1,...,M\}$ the following conditions are satisfied:
    \begin{enumerate}
        \item $\mathcal{C}^{(j)}$ is an $(h^{(j)}, \nu^{(j)}, n^{(j)})$-cluster.
        \item $\|x-y\|_{\mathbb{T}} \geq \eta > 0$, $\forall x \in \mathcal{C}^{(j)}, \forall y \in X \setminus \mathcal{C}^{(j)}$.
    \end{enumerate}
\end{definition}
\begin{definition}\label{def:van}
    For a finite set of nodes $X$ and sampling set $S$, we define the corresponding \textbf{Vandermonde} matrix as 
    \[
        V(X;S) = \big[e^{ikx}\big]_{k \in S}^{x\in X}.
    \]
\end{definition}
Let $\sigma_1(L) \geq \sigma_2(L) \geq \dots \geq \sigma_{n}(L)$ denote the singular values of a matrix $L \in \mathbb{C}^{n\times n}$ and let $\lambda_1(L) \geq \lambda_2(L) \geq \dots \geq \lambda_n(L)$ denote it's eigenvalues.

\section{Identification of an Optimal Rate}
In this section, we establish a measure for approximating the minimal separation $\Delta_{\rho}$ between decimated nodes with decimation parameter $\rho$. For $M$ clusters, we prove in proposition \ref{pro:teop-eig} that $\sigma_{M+1}(T_{\rho}) \asymp \Delta_{\rho}^2$, where $\asymp$ denotes asymptotic scaling. 
\subsection{Main result and proofs}
Let $\hat{\mu}_k := \hat{\mu}(k) = \sum_{j=1}^na_je^{ikx_j}$ be the Fourier measurement of model \eqref{spike}.
For a set of nodes $X := \{x_1,...,x_n\}$, the Toeplitz matrix of the samples is defined as:
\begin{equation}\label{teoplitz}
        T = T(X;n) := \begin{bmatrix}
            \mu_{n-1} & \mu_{n-2} & \dots & \mu_{0} \\
            \vdots & \vdots &  & \vdots \\
            \mu_{2n-1} & \mu_1 & \dots & \mu_{n-1} \\
        \end{bmatrix}.
\end{equation}
and it can be decomposed into:
\begin{align*}
    T = \tilde{V}_{n}\underbrace{diag(a_1,\dots,a_n)}_{A}V_n^*,
\end{align*}
where 
\begin{align*}\label{eq:van}
    &\tilde{V}_{n} := V(X;\{n-1,\dots,2n-2\}), \\ &V_n := V(X;\{0,\dots,n-1\}).
\end{align*}

We have 
\[ 
    \tilde{V}_n := V_ndiag(e^{i(n-1)x_1},...,e^{i(n-1)x_n}) = V_nE,
    \]
then we can write 
\begin{equation}\label{eq:sq-van}
    T = V_nDV_n^*, \quad D=EA.
\end{equation}

\begin{theorem}\label{thm:main}
    Let $Q \in \mathbb{C}^{n\times n}$ be a matrix of the form $Q = V^*DV$, where $D$ is a diagonal complex matrix and $V \in \mathbb{C}^{n\times n}$. Then we have:
    \begin{equation}
        \sigma_i(Q) = \big|\lambda_i(Q)\big| = \theta_i\lambda_i(VV^*), \quad |D| := (DD^*)^{1\over 2}.
    \end{equation}
    where $\lambda_n((DD^*)^{1\over 2}) \leq \theta_i \leq \lambda_1((DD^*)^{1\over 2})$.
\end{theorem}

\begin{proof}
    \begin{align*}
        \big|\lambda_i(V^*DV)\big|^2 &= \lambda_i(V^*DV) \bar{\lambda}_i(V^*DV)  \\&= \lambda_i(V^*DV)\lambda_i(V^*D^*V).
    \end{align*}
    from properties of eigenvalues:
    \begin{align*}
        &\mu_i := \lambda_i(V^*DV) = \lambda_i(VV^*D) = \lambda_i(D^{1\over 2}VV^*D^{1\over 2}), \\
        &\bar{\mu}_i := \lambda_i(V^*D^*V) = \lambda_i(VV^*D^*) = \lambda_i((D^*)^{1\over 2}VV^*(D^*)^{1\over 2}).
    \end{align*}
    and thus 
    \begin{align*}
        &D^{1\over 2}VV^*D^{1\over 2}u_i = \mu_iu_i, \\
        &u^*_i(D^*)^{1\over 2}VV^*(D^*)^{1\over 2} = \bar{\mu}_i u^*_i.
    \end{align*}
    Note that these matrices are not necessarily Hermitian.
    \begin{align*}
        &u^*_i(D^*)^{1\over 2}VV^*(D^*)^{1\over 2}D^{1\over 2}VV^*D^{1\over 2}u_i = \bar{\mu}_i\mu_i u^*_iu_i, \\
        &u^*_i(D^*)^{1\over 2}VV^*|D|VV^*D^{1\over 2}u_i = \bar{\mu}_i\mu_i u^*_iu_i, \\
        &u^*_i(D^*)^{1\over 2}VV^*D^{1\over 2}(D^*)^{1\over 2}VV^*D^{1\over 2}u_i = \bar{\mu}_i\mu_i u^*_iu_i, \\
        &\frac{\|(D^*)^{1\over 2}VV^*D^{1\over 2}u_i\|_2^2}{\|u_i\|_2^2} = \mu_i\bar{\mu}_i = \sigma_i^2((D^*)^{1\over 2}VV^*D^{1\over 2}), \\
        &\sigma_i((D^*)^{1\over 2}VV^*D^{1\over 2}) = \lambda_i((D^*)^{1\over 2}VV^*D^{1\over 2}).
    \end{align*}
    Where the last equivalence is true because $(D^*)^{1\over 2}VV^*D^{1\over 2}$ is a positive semi-definite hermitian matrix.
    By Ostrowski's Theorem for Hermitian matrices (Theorem 4.5.9 in \cite{horn2012matrix}), we have
    \[ \lambda_i(D^{1\over 2}VV^*(D^*)^{1\over 2}) = \theta_i\lambda_i(VV^*), \]
    where 
        \[
            \lambda_n((DD^*)^{1\over 2}) \leq \theta_i \leq \lambda_1((DD^*)^{1\over 2}).
        \]
    Lastly we have, $\big|\lambda_i(V^*DV)\big| = \sqrt{\lambda_i(V^*DV)\cdot \bar{\lambda}_i(V^*DV)} = \sqrt{\lambda_i(V^*D^*VV^*DV)} = \sigma_i(V^*DV)$.
\end{proof}

From now on, we choose $X$ to form a $((h^{(j)}, \nu^{(j)}, n^{(j)})_{j=1}^M, \eta)$-clustered configuration, as defined in Definition~\eqref{def:multi-cluster}, with $h^{(1)} = h^{(2)} = \dots = h^{(M)} = \Delta(X)$.

\begin{theorem}\label{thm:lower-bnd}
    There exist $C_1, C_2$ depending on $s:= \max_j n^{(j)}$ such that for $\eta \geq C_1$ and $\Delta \leq {C_2 \over n^2\nu}, \nu := \max_k \nu^{(k)}$, for each $m=1,...,s$ there are precisely $\ell_m := \# \{1\leq k \leq M : m \leq n^{(k)}\}$ singular values of $V_n:=V_n(X)$ bounded below by
    \[
        C\Delta^{m-1},
    \]
    where $C$ doesn't depend on $\Delta$.
\end{theorem}
\begin{proof}
    Recall the following Theorem:
    \begin{theorem}[Proposition 7.1 in \cite{batenkov2021single}]
        Let $\tilde{\sigma}_1 \geq \tilde{\sigma}_2 \geq \dots \geq \tilde{\sigma}_n$ denote the union of all the singular values of the matrices $V_n^{(j)}:= V(\mathcal{C}^{(j)};\{0,\dots,n-1\})$ in non-increasing order, and ${\sigma}_1 \geq {\sigma}_2 \geq \dots \geq {\sigma}_n$ denote the singular values of $V_n$. Then there exist $C_1(s), C_2(s)$ depending only on $s:= \max_j{n^{(j)}}$ such that for $\eta \geq C_1$ and $\Delta \leq {C_2 \over n^2\nu}, \nu := \max_k \nu^{(k)}$ we have 
        \[
            \sigma_j \geq {1\over 2}\tilde{\sigma}_j, \quad j=1,...,n.
        \]
    \end{theorem}
    We have to find the scaling of the singular values of each $V_n^{(j)}$ to finish the proof.
    For simplicity, assume that $n=2p+1$.
    Thus we can write,
    \[
        U_p^{(j)} = \big[e^{ikx}\big]_{k \in \mathcal{G}_p}^{x\in \mathcal{C}^{(j)}}.
    \]
    where $D_j:= diag(e^{-ipx})_{x \in \mathcal{C}^{(j)}} \in \mathbb{C}^{n^{(j)}\times n^{(j)}}$ and $\mathcal{G}_p:=\{-p,-p+1,...,0,...,p\}$ is the symmetric one-dimensional grid. Let $G_j := \big(U_p^{(j)}\big)^*U_p^{(j)}$ for $j=1,...,M$. Now we use the following Lemma and write it for the uni-variate case:
    \begin{lemma}[Lemma 3 in \cite{diab2024spectral}]
         Let $V_{\leq m}(\mathcal{G}) := [g^k]_{g\in \mathcal{G}}^{0\leq k \leq m}$. Then the sampling set $\mathcal{G}_p$ satisfies the condition $rank(V_{\leq n^{(j)} - 1}(\mathcal{G}_p)) \geq n^{(j)}$ and the eigenvalues of $G_j$ split into $n^{(j)}$ groups:
         \begin{align*}
             &\lambda_0 = \Delta^0(\tilde{\lambda}_0 + O(\Delta)), \quad \lambda_1 = \Delta^2(\tilde{\lambda}_1 + O(\Delta)), ..., \\ & \lambda_{n^{(j)}} = \Delta^{2(n^{(j)}-1)}(\tilde{\lambda}_{n^{(j)}} + O(\Delta)).
         \end{align*}
         where $\tilde{\lambda}_i$ doesn't depend on $\Delta$.
    \end{lemma}
    Since $\big(V_n^{(j)}\big)^*V_n^{(j)} = (D_j^{-1})^*G_jD_j^{-1}$, we finish the proof of Theorem \ref{thm:lower-bnd}.
\end{proof}
\begin{theorem}\label{thm:gram-eig}
    The singular values of the Vandermonde matrix $V_n$ scale as follows:
    \begin{align*}
        &\{\sigma_{1,i}(V_n)\}_{i=1}^{\ell_1} \asymp {\Delta}^{0}, \dots, \{\sigma_{s,i}(V_n)\}_{i=1}^{\ell_s} \asymp {\Delta}^{s-1},
    \end{align*}
    where $s := \max_k n^{(j)}$. 
\end{theorem}
\begin{proof}
    The determinant of the Vandermonde matrix $V_n$ is: 
    \[
        det(V_n) = \prod_{k\neq j}|e^{ix_j} - e^{ix_k}|.
    \]
    Using Lemma 5.6 in \cite{Batenkov2021}, for $|x - x'| \leq {\pi \over 2}$ we have 
    \[
        {2\over \pi}|x - x'| \leq |e^{ix} - e^{ix'}| \leq |x - x'|.
    \]
    Thus we get that
    \begin{align*}
        det(V_n) &= \prod_{{\scriptscriptstyle j=1}}^{\scriptscriptstyle M}\prod_{{\scriptscriptstyle x, y \in \mathcal{C}^{(j)}}}|e^{ix} - e^{iy}|\prod_{{\scriptscriptstyle  j \neq i}}\prod_{{\scriptscriptstyle x \in \mathcal{C}^{(j)}, y \in \mathcal{C}^{(i)}}}|e^{ix} - e^{ix'}|  \\ &= C(\eta,\nu,n)\prod_{{\scriptscriptstyle j=1}}^{\scriptscriptstyle M}\prod_{{\tiny x, y \in \mathcal{C}^{(j)}}}|e^{ix} - e^{iy}| = C\Delta^{\sum_{j=1}^M {n^{(j)}\choose 2}}.
    \end{align*}
    Now we show that
    \begin{align*}
        \sum_{j=1}^M {n^{(j)}\choose 2} = \sum_{j=1}^M {n^{(j)}(n^{(j)}-1) \over 2} &= \sum_{j=1}^M \sum_{i=1}^{n^{(j)}}(i-1) \\ &= \sum_{i=1}^s \ell_i(i-1).
    \end{align*}
    where the last step is due to the definition of $\ell_j$.
    Thus we have that $det(V_n) = C\cdot \Delta^{\sum_{i=1}^s \ell_i(i-1)}$. It's known that $|det(V_n)| = \prod_{i=1}^n \sigma_i(V_n)$. Thus on the one hand we have $\prod_{i=1}^n \sigma_i(V_n) = C\cdot \Delta^{\sum_{i=1}^s \ell_i(i-1)}$ and on the other hand (by Theorem \ref{thm:lower-bnd}) we have that
    \begin{align*}
        \prod_{i=1}^n \sigma_i(V_n) \geq \prod_{i=1}^s \prod_{k=1}^{\ell_k}C_{i,k}\sigma_{i,k}(V_n) &= \prod_{i=1}^s\prod_{k=1}^{\ell_k}C_{i,k}\Delta^{i-1} \\ &= \tilde{C}\Delta^{\sum_{i=1}^s \ell_i(i-1)}.
    \end{align*}
    Since $C$ doesn't depend on $\Delta$, this forces the singular values to have exact scaling (i.e. $C_{i,k}$ don't depend on $\Delta$).
\end{proof}
Now we state our main result.
\begin{proposition}\label{pro:teop-eig}
     For any $\rho > 0$, the singular values of the Toeplitz matrix $T_{\rho}:=T(\rho X;n)$ scales as follows:
    \begin{align*}
        \{\sigma_{1,i}(T_{\rho})\}_{i=1}^{\ell_1} \asymp {\Delta}_{\rho}^{0}, \dots , \{\sigma_{s,i}(T_{\rho})\}_{i=1}^{\ell_s} \asymp {\Delta}_{\rho}^{2(s-1)},
    \end{align*}
    where $s := \max_j n^{(j)}$ and $\Delta_{\rho} := \Delta(\rho X)$.
\end{proposition}
\begin{proof}
    Combining Theorems \ref{thm:main} and \ref{thm:gram-eig}, we get the desired result.
\end{proof}
\begin{remark}
    The constants in Proposition \ref{pro:teop-eig} also depend on the amplitudes $\{a_j\}_{j=1}^n$.
\end{remark}
\subsection{Numerical Validation}
Let $k \in \mathbb{N}$ be the number of clusters in a configuration with nodes in $X$. To validate Proposition \ref{pro:teop-eig}, we plotted the minimal separation of decimated nodes for all decimation parameters $\rho$ in the interval $\mathcal{I} := \big[{1\over 2}{\Omega \over 2n-1},{\Omega \over 2n-1}\big]$ against the $k+1$-th singular value of the Toeplitz matrix $T_{\rho}$. Figure \ref{fig:teo_sing} confirms that $\sigma_{k+1}(T_{\rho})$ scales proportionally to $\Delta_{\rho}^2$. 
\begin{figure}[htb]
  \includegraphics[width=0.49\linewidth]{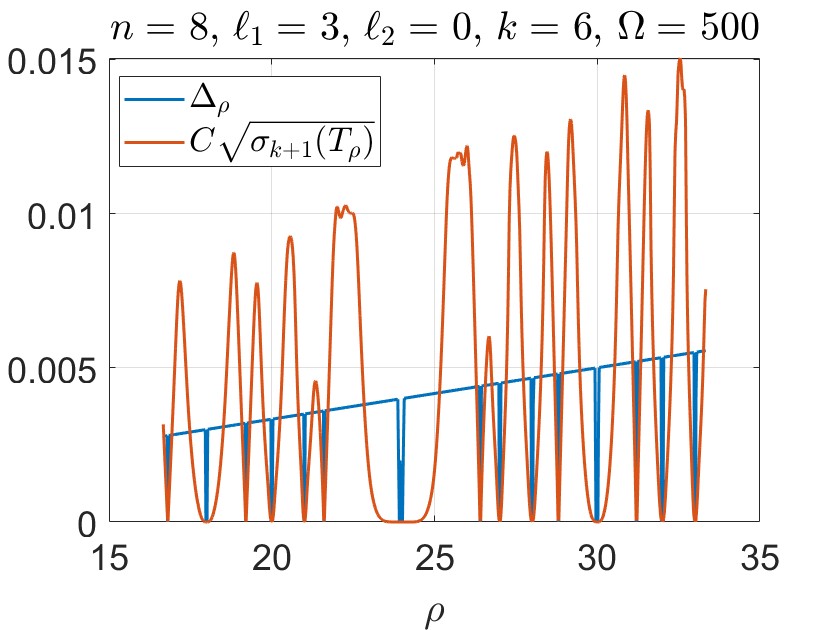} \hfill
  \includegraphics[width=0.49\linewidth]{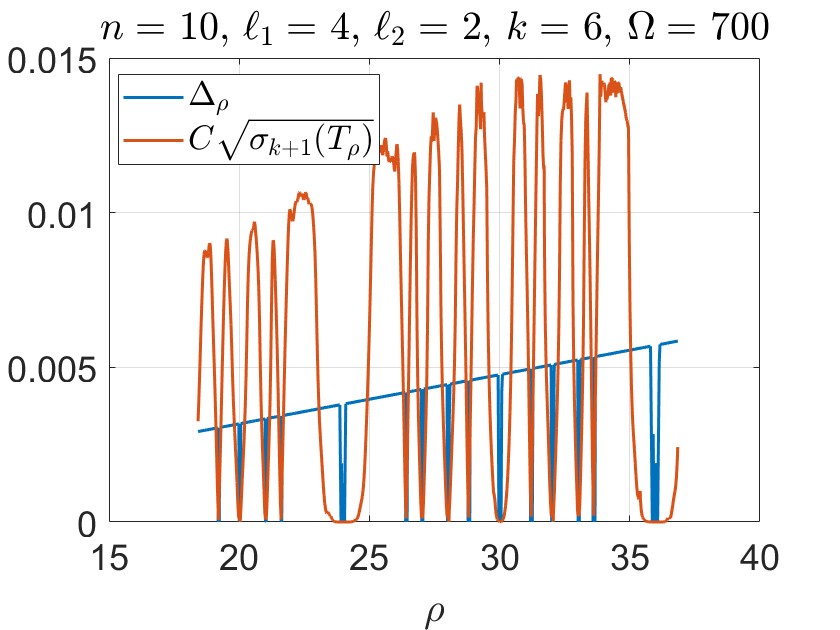}
  \caption{\small  (right) $X$ has one cluster of size $\ell_1$ with $SRF = 6$. (left) $X$ has two clusters of sizes $\ell_1$ and $\ell_2$ with $SRF = 3$. The plotted $\sqrt{\sigma_{k+1}(T_{\rho})}$ values are scaled by $C := \frac{n}{\Omega}$ for better visualization.
}
  \label{fig:teo_sing}
\end{figure}
\begin{remark}
    We expect our estimates in Proposition \ref{pro:teop-eig} to hold for noisy samples, using standard perturbation analysis for singular values. We leave it for future work.
\end{remark}

\section{Algorithm}
A decimation parameter $\lambda \in \mathcal{I}$ is said to be admissible if it attains good separation properties, in particular, when the cluster nodes are separated by at least $O(\Omega\Delta)$ and the non-cluster nodes by $\geq {1\over 2n^2}$ (for the full details, see Proposition 5.8 in \cite{Batenkov2021}).

\subsection{Selecting an optimal sampling rate}
By Proposition \ref{pro:teop-eig}, for any \(\rho \in \mathbb{N}\), we have \(\sigma_{M+1}(T_{\rho}) \asymp \Delta_{\rho}^2\), where \(M\) is the number of clusters with multiplicity at least 1. For a set \(\Lambda = \{\rho_j,\}_{j=1}^m \subset \mathbb{N}\), we select $\rho = \arg\max_{j}\{\sigma_{M+1}(T_{\rho_j})\}_{j=1}^m$
as the optimal decimation parameter with the best separation properties.

Since decimation alone is insufficient for recovery, any SR method applied to samples at rate \(\rho\) gives nodes \(\{e^{i\rho x_j}\}_{j=1}^n\), with each \(e^{i\rho x_j}\) having \(\rho\) candidates. To resolve this ambiguity, we use the technique from \cite{cuyt2020get} for de-aliasing, which employs a second shifted sample set \(\mathcal{M}_{ds}:=\{\mu(\rho k+ t)\}_{k=0}^{2n-1}\) alongside \(\mathcal{M}_{d}:=\{\mu(\rho k)\}_{k=0}^{2n-1}\), where \(\rho\) and \(t\) are co-prime. 
Let $\Phi_j := \exp{(ix_j)}$. In the noiseless case, we have $\mu(\rho k) = \sum_{j=1}^n a_j \Phi_j^{\rho k}$ and $\mu(\rho k + t) = \sum_{j=1}^n (a_j \Phi_j^t) \Phi_j^{\rho k}$.
To recover $\{\Phi_j^t\}_{j=1}^n$, we first recover the amplitudes $\{a_j\}_{j=1}^n$ and $\{a_j\Phi_j^t\}_{j=1}^n$ from the sampling sets $\mathcal{M}_{d}$ and $\mathcal{M}_{ds}$, respectively. Then compute $\{\Phi_j^t\}_{j=1}^n = \left\{\frac{a_j\Phi_j^t}{a_j} \right\}_{j=1}^n$.
Finally, since $\rho$ and $t$ are co-prime, the intersection of the candidate sets for $e^{i\rho x_j}$ and $e^{it x_j}$ contains exactly one element. For noisy samples, we first match the aliased nodes \( \{\tilde{\Phi}_j^{\rho}\} \) and \( \{\hat{\Phi}_j^{\rho}\} \) before division. A more stable solution can be achieved using additional shifted sample sets; see \cite{cuyt2020get, Briani2020} for the full details.
\par Summarizing the above, Algorithm \ref{alg:decimation} can be applied to any SR method to generate its decimated version.

\begin{algorithm}
\caption{Decimated SR method}
\label{alg:decimation}
\begin{algorithmic}[1]
\Require \(M\), \(\{\hat{\mu}(\omega)\}_{\omega \in [-\Omega,\Omega]}\), \(\Omega\), $n$, SRmethod($\cdot$).
\State Select Admissible Decimation parameter:
\State \hspace{1em} (a) Select $\rho := \arg\max_{\rho' \in \mathcal{I}\cap\mathbb{N}}\{\sigma_{M+1}(T_{\rho'})\}$.
\State \hspace{1em} (b) Find \(t\) such that \(\rho\) and \(t\) are co-prime.
\State Solve for Decimated and Shifted samples:
\State \hspace{1em} (a) $\{\tilde{\Phi}_j^{\rho}, \tilde{a}_j\}_{j=1}^n =$ SRmethod($\mathcal{M}_d$).
\State \hspace{1em} (b) $\{\hat{\Phi}_j^{\rho}, \hat{a}_j\hat{\Phi}_j^{t}\}_{j=1}^n =$ SRmethod($\mathcal{M}_{ds}$).
\State Perform De-aliasing for $\{\tilde{\Phi}_j^{\rho},\hat{\Phi}_j^{t}\}$ (see the text above).
\State Return the estimates $\{\tilde{x}_j, \tilde{a}_j\}$.

\end{algorithmic}
\end{algorithm}

\begin{remark}
    Using Proposition 5.8 from \cite{Batenkov2021}, we can derive a lower bound of \( \frac{1}{n^2} \) for selecting an admissible decimation parameter from \( \mathcal{I} \), enabling randomness in algorithm \ref{alg:decimation}.
\end{remark}

We introduce the Enhanced Decimated Prony Method (EDP) and the Decimated Matrix Pencil method (DMP) as improvements over DP and MP, respectively. EDP applies Algorithm \ref{alg:decimation} with Prony's method as the SR method, eliminating the need to solve multiple sub-Prony problems and constructing a histogram, as done in DP. Similarly, DMP applies Algorithm \ref{alg:decimation} with MP as the SR method.

\par DP and MP have a time complexity of $O(\Omega^2) + O(\Delta^{-1})$ and $O(\Omega^3)$, respectively. For EDP and DMP, there are $O(\Omega)$ natural decimation parameters in the interval $\mathcal{I}$. For each, we construct an $n \times n$ Toeplitz matrix of samples and compute its $M+1$-th singular value, which requires $O(n^3\Omega)$. Finding a co-prime $t \neq 1$ for $\rho$ has a worst-case complexity of $O(\Omega \log(\Omega))$, as gcd computations take $O(\log(\Omega))$, and the search might iterate up to $\Omega - 1$. Since matching the aliased nodes, applying Prony's method, MP (for $2n$ samples) and de-aliasing depend only on $n$ (fixed), the total time complexity of EDP and DMP is $O(\Omega \log(\Omega))$. 

\par We implemented the EDP, DP, DMP and MP methods, assigning the number of nodes for both MP and DMP as an input and using $3n$ decimated samples in DMP. The implementation is in MATLAB. We recall that for DP, $N_{\rho}$ is the number of decimation parameters being tested in the interval $\mathcal{I}$ and $N_b$ is the number of bins in the constructed histogram of node candidates. The noise in the samples is generated from a Cauchy distribution. Each method was tested $10$ times across four different SRF values. In Figure \ref{fig:time-com}, we present the mean reconstruction error of the cluster node $x_1$ as a function of SRF. The results show that all methods achieve high accuracy, with EDP and DMP being the fastest, offering a speed improvement of up to two orders of magnitude.

\begin{figure}[htb]
    \centering
  \includegraphics[width=0.49\linewidth]{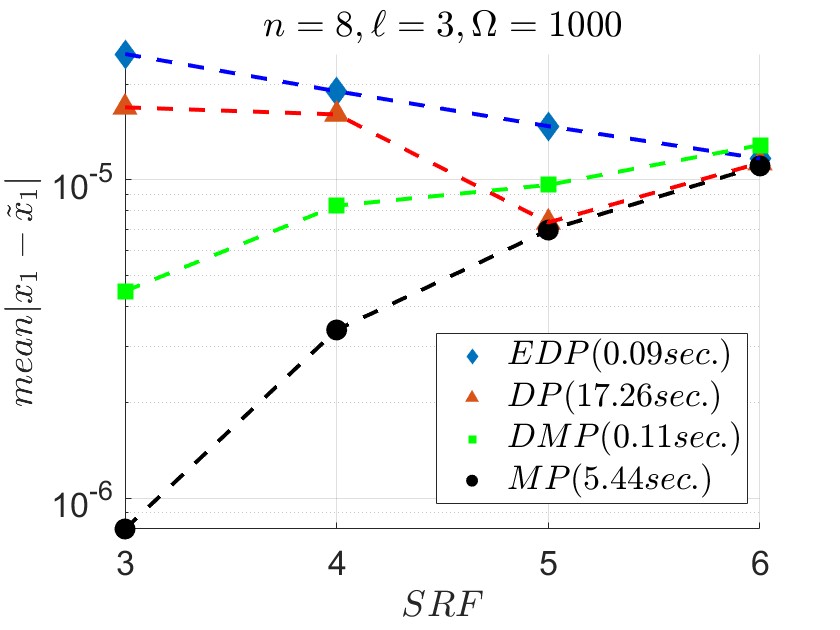} \hfill
  \includegraphics[width=0.49\linewidth]{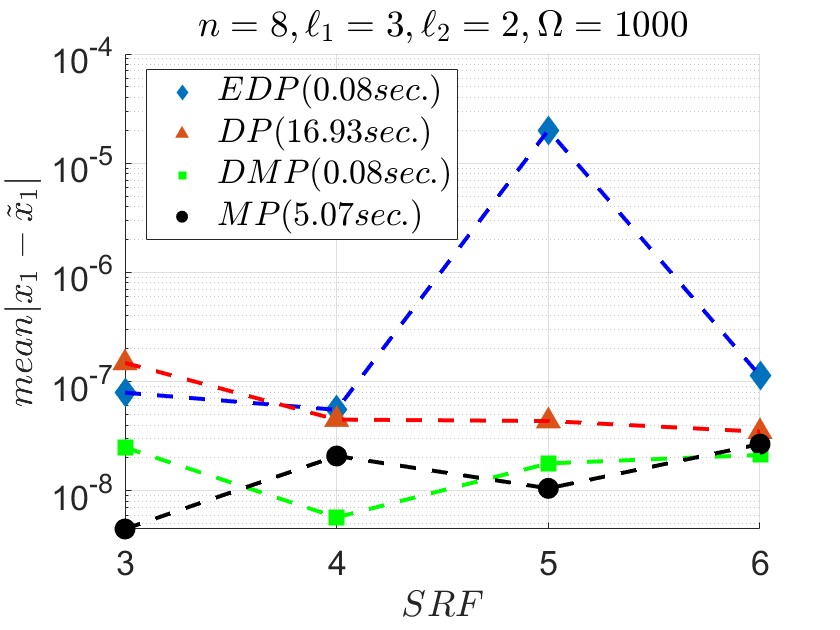}
  \caption{\small (right) single cluster configuration. (left) multi-cluster configuration. For both experiments the noise level is $10^{-6}$, $N_{\rho} = 900$ and $N_b = 3\Delta^{-1}$.}
  \label{fig:time-com}
\end{figure}
\begin{remark}
    To obtain higher accuracy for DMP, we can use $N \gg 2n$ decimated samples if possible.
\end{remark}
\subsection{Optimality of EDP}
%
%
In Figure \ref{fig:opt}, we numerically show that EDP is optimal, meaning it achieves the min-max error bounds (Theorem 2.8 in \cite{Batenkov2021}) in the multi-cluster geometry. We plot the node/amplitude error amplification factors \eqref{naf} as a function of SRF.
\begin{equation}\label{naf}
    \mathcal{K}_{x_j} := \epsilon^{-1}\Omega |x_j - \tilde{x}_j|, \quad \mathcal{K}_{a_j} := \epsilon^{-1} |a_j - \tilde{a}_j|.
\end{equation}


\begin{figure}[htb]
  \includegraphics[width=0.49\linewidth]{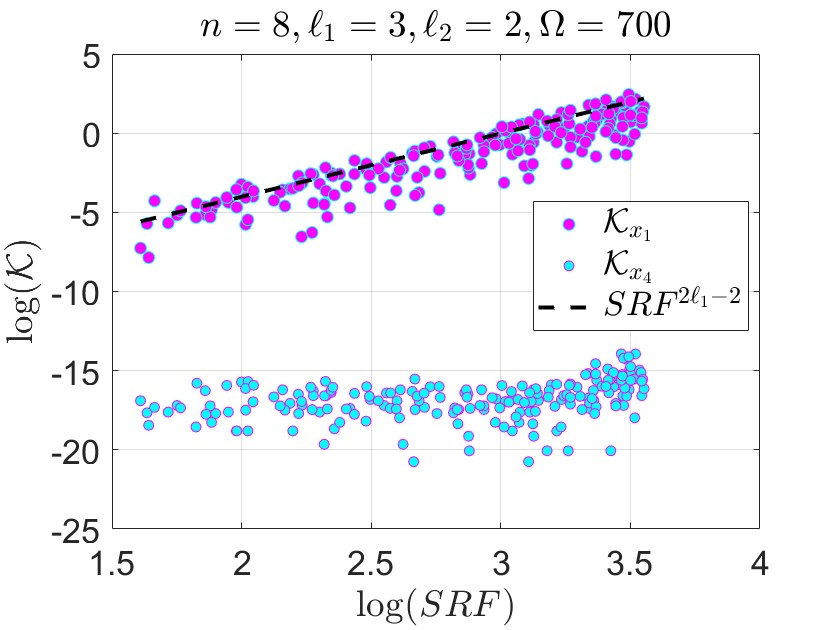} \hfill
  \includegraphics[width=0.49\linewidth]{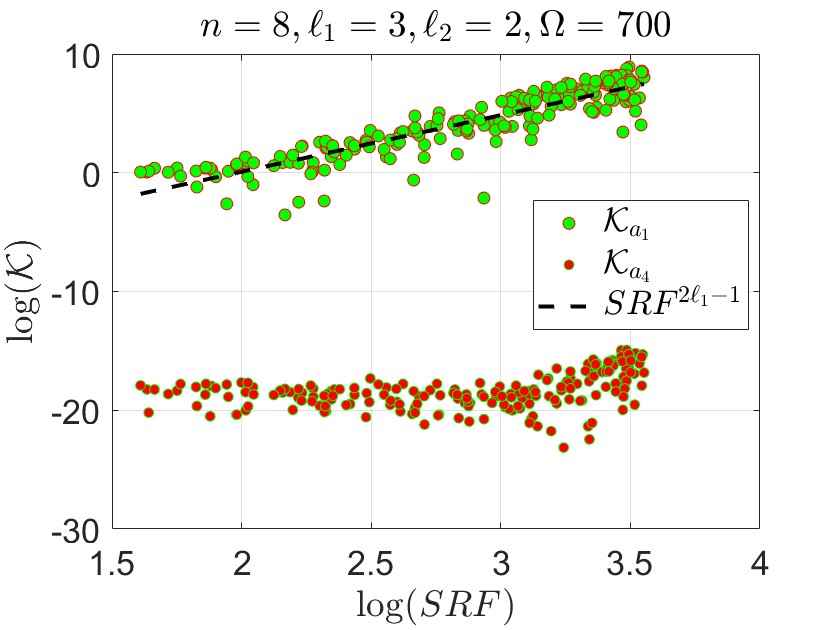}
  \caption{\small EDP - asymptotic optimality. For cluster node $x_1$, $\mathcal{K}_{x_1}$ (left) scales
                like $SRF^{2\ell_1 - 2}$, while the $\mathcal{K}_{a_1}$ (right) scales like $SRF^{2\ell_1 - 1}$.
                For the non-cluster node $x_4$, both $\mathcal{K}_{x_4}$ and $\mathcal{K}_{a_4}$ are lower bounded
                by a constant. These scaling rates are optimal.}
  \label{fig:opt}
\end{figure}

\textbf{Acknowledgment}. We thank Dr. Dmitry Batenkov and Prof. Ronen Talmon for their insightful comments.
\clearpage
\bibliographystyle{plain}
\bibliography{ref}
\end{document}